\documentclass[12 pt]{amsart}

\usepackage{amsmath} 
\usepackage{amssymb}
\usepackage{latexsym}
\usepackage[mathscr]{eucal}
\newtheorem{theorem}{Theorem}

\newtheorem{conjecture}[theorem]{Conjecture}
\newtheorem{lemma}[theorem]{Lemma}

 \def\Ind{\operatorname{Ind}}
 \def\diag{\operatorname{diag}}
 \def\Jord{\operatorname{Jord}}
 \def\im{\operatorname{Im}}

\begin{document}

\title{$R$-groups and parameters}

\author{Dubravka Ban}
\author{David Goldberg}

\address{Department of Mathematics\\
Southern Illinois University\\ Carbondale, IL 62901 \\USA }

\email{dban@math.siu.edu}

\address{Department of Mathematics\\ Purdue University\\
West Lafayette, IN 47907\\ USA}

\email{goldberg@math.purdue.edu}

\thanks{D.B. supported in part by NSF grant DMS-0601005 and by a Research Fellowship
of the Alexander von Humboldt Foundation.}

\maketitle

\section{Introduction} 

Central to representation theory of reductive groups over local fields is the study of parabolically induced representations.  In order to classify the tempered spectrum of such a group, one must understand the structure of parabolically induced from discrete series representations,  in terms of components, multiplicities, and and whether or not components are elliptic.  The Knapp-Stein $R$--group gives an explicit combinatorial method for conducting this study.  On the other hand, the local Langlands conjecture predicts the parameterization of such non-discrete tempered representations, in $L$--packets, by admissible homomorphisms of the Weil-Deligne group which  factor through a Levi component of the Langlands dual group. In \cite{Art1}, Arthur gives a conjectural description of the Knapp-Stein $R$--group in terms of the parameter.  This conjecture generalizes results of Shelstad, \cite{Shel}, for archimedean groups, as well as those of Keys, \cite{keys}, in the case of unitary principal series of certain $p$--adic groups.  In \cite{BZ} the first named author and Zhang establish this conjecture for odd special orthogonal groups.  In \cite{G3}, the second named author establishes the conjecture for induced from supercuspidal representations of split special orthogonal or symplectic groups, under an assumption on the parameter.  In the current work, we complete the conjecture for the full tempered spectrum of all these groups.

Let $F$ be a  nonarchimedean local field of characteristic zero.  We denote by $\bold G$  a connected reductive quasi-split algebraic group defined over $F.$  We let $G=\bold G(F),$ and use similar notation for other groups defined over $F.$  Fix a maximal torus $\bold T$ of $\bold G,$ and a Borel subgroup $\bold B=\bold T\bold U$ containing $\bold T.$    We let $\mathcal E(G)$ be the equivalence classes of irreducible admissible representations of $G,$ let $\mathcal E_t(G)$ be the tempered classes, let $\mathcal E_2(G)$ be the discrete series, and let  $^\circ\mathcal E(G)$ the irreducible unitary supercuspidal classes.   We make no distinction between a representation $\pi$ and its equivalence class.

Let $\bold P=\bold M\bold N$ be a standard, with respect to $\bold B,$ parabolic subgroup of $\bold G.$ Let $\bold A=\bold A_{\bold M}$ be the split component of $\bold M,$ and let $W=W(\bold G,\bold A)=N_{\bold G}(\bold A)/\bold M$ be the Weyl group for this situation.  For  $\sigma\in\mathcal E(M)$ we let $\Ind_P^G(\sigma)$ be the representation unitarily induced from $\sigma\otimes\bold 1_N.$ Thus, if $V$ is the space of $\sigma,$ we let
$V(\sigma)=\left\{f\in C^\infty(G,V)\,\big | f(mng)=\delta_P(m)^{1/2}f(g),\,\forall m\in M,n\in N,\text{ and }g\in G\right\},$
with $\delta_P$ the modulus character of $P.$  The action of $G$ is by the right regular representation, so $(\Ind_P^G(\sigma)(x)f)(g)=f(gx).$  Then any $\pi\in\mathcal E_t(G)$ is an irreducible component of $\Ind_P^G(\sigma)$ for some choice of $M$ and $\sigma.$   In order to determine the component structure of $\Ind_P^G(\sigma),$ Knapp and Stein, in the archimedean case, and Harish-Chandra in the $p$-adic case, developed the theory of singular integral intertwining operators, leading to the theory of $R$-groups, due to Knapp and Stein in the archimedean case, \cite{Kn-St}, and Silberger  in the $p$-adic case \cite{Sil1,Sil2}.  We describe this briefly and refer the reader to the introduction to \cite{G1} for more details.  The poles of the intertwining operators give rise to the zeros of Plancherel measures.
Let $\Phi(\bold P,\bold A)$ be the reduced roots of $\bold A$ in $\bold P.$  For $\alpha\in\Phi(\bold P,\bold A)$ and $\sigma\in\mathcal E_2(M)$ we let $\mu_\alpha(\sigma)$ be the rank one Plancherel measure associated to $\sigma$ and $\alpha.$  We let $\Delta'=\{\alpha\in\Phi(\bold P,\bold A)\,|\, \mu_\alpha(\sigma)=0\}.$    For $w\in W$ and $\sigma\in\mathcal E_2(M)$ we let $w\sigma(m)=\sigma(w^{-1}m\sigma).$  (Note, we make no distinction between $w\in W$ and its representative in $N_G(A).)$   We let $W(\sigma)=\{w\in W|w\sigma\simeq\sigma\},$ and $W'$ be the subgroup of $W(\sigma)$ generated by those $w_\alpha$ with $\alpha\in\Delta'.$  We let $R(\sigma)=\{w\in W(\sigma)|w\Delta'=\Delta'\}=\{w\in W(\sigma)|w\alpha>0,\forall\alpha\in\Delta'\}.$   Let $\mathcal C(\sigma)=\operatorname{End}_G(\Ind_P^G(\sigma)).$

\begin{theorem} (Knapp-Stein, Silberger,\cite{Kn-St,Sil1,Sil2})
For any $\sigma\in\mathcal E_2(M),$
we have $W(\sigma)= R(\sigma)\ltimes W',$ and $\mathcal C(\sigma)\simeq\Bbb C[R(\sigma)]_\eta,$ the group algebra of $R(\sigma)$ twisted by a certain  2-cocycle
$\eta.$
\end{theorem}

Thus $R(\sigma),$ along with $\eta,$ determines how many inequivalent components appear in $\Ind_P^G(\sigma)$ and the multiplicity with which each one appears.  Furthermore Arthur shows $\Bbb C[R(\sigma)]_\eta$ also determines whether or not components of $\Ind_P^G(\sigma)$ are elliptic (and hence whether or not they contribute to the Plancherel Formula) \cite{Art2}.

In \cite{Art1} Arthur conjectures a construction of $R(\sigma)$ in terms of the local Langlands conjecture.  Let $W_F$ be the Weil group of $F,$ and $W_F'=W_F\times SL_2(\Bbb C)$ be the Weil-Deligne group.  Suppose $\psi:W_F'\rightarrow\,^LM$ paramterizes the $L$--packet,  $\Pi_\psi(M),$ of $M$ containing $\sigma.$  Here $^LM=\hat M\rtimes W_F$ is the Langlands $L$--group, and $\hat M$ is the complex group whose root datum is dual to that of $\bold M.$  Then $\psi:W_F'\rightarrow\,^LM\hookrightarrow\,^LG$ must be a parameter for an $L$--packet, $\Pi_\psi(G), $ of $G.$  The expectation is $\Pi_\psi(G)$ consists of all irreducible components of $\Ind_P^G(\sigma')$ for all $\sigma'\in\Pi_\psi(M).$  We let $S_\psi=Z_{\hat G}(\im \psi),$ and take $S_\psi^\circ$ to be the connected component of the identity.  Let $T_\psi$ be a maximal torus in $S^\circ_\psi.$  Set $W_\psi=W(S_\psi,T_\psi),$ and $W_\psi^\circ=W(S_\psi^\circ,T_\psi).$  Then $R_\psi=W_\psi/W_\psi^\circ$ is called the $R$-group of the packet $\Pi_\psi(G).$
By duality we can identify $W_\psi$ with a subgroup of $W.$  With this identification, we let $W_{\psi,\sigma}=W_\psi\cap W(\sigma)$ and $W_{\psi,\sigma}^\circ=W_\psi^\circ\cap W(\sigma).$  We then set $R_{\psi,\sigma}=W_{\psi,\sigma}/W_{\psi,\sigma}^\circ.$  We call $R_{\psi,\sigma}$ the Arthur $R$-group attached to $\psi$ and $\sigma.$

\begin{conjecture} For any $\sigma\in\mathcal E_2(M),$ we have $R(\sigma)\simeq R_{\psi,\sigma}.$
\end{conjecture}

In \cite{BZ}, the first named author and Zhang proved this conjecture in the case $\bold G=SO_{2n+1}.$ In \cite{G3} the second named author confirmed the conjecture when $\sigma$ is supercuspidal, and $\bold G= SO_{n} or Sp_{2n},$ with a mild assumption on the parameter $\psi.$  Here, we complete the proof of the conjecture for $\bold Sp_{2n},$ or $O_n$, under assumptions given in Section~\ref{sect1.3}.

We now describe the contents of the paper in more detail.  In Section 2 we introduce our notation and discuss the classification of $\mathcal E_2(M)$ for our groups, due to M\oe glin and Tadi\'c, as well as preliminaries on Knapp-Stein and Arthur $R$--groups. In Section 3 we consider the parameters $\psi$ and compute their centralizers. In Section 4 we turn to the case of $\bold G=O_{2n}.$  Here we show the Arthur $R$--group agrees with the generalization of the Knapp Stein $R$--group as discussed in \cite{GH}.  In Section 5 we complete the proof of the Theorem for the induced from discrete series representations of $Sp_{2n}, SO_{2n+1},$ or $O_{2n}.$  

In Section 6, we study $R$-groups for unitary groups. These groups are interesting for us 
because they are not split and the action of the Weil group on the dual group is nontrivial.
In addition, the classification of discrete series and description of $L$-parameters
is completed \cite{Moe3}.

The techniques used here can be used for other groups.  In particular we should be able to carry out this process for  similitude groups  and $G_2.$  Furthermore, the techniques of computing 
the Arthur $R$--groups will apply to $GSpin$ groups, as well, and may shed light on the Knapp-Stein $R$--groups in this case.  We leave all of this for future work.

By closing the introduction, we would like to thank Guy Henniart, Joe Hundley and Freydoon Shahidi
for valuable comments. 
D.B. thanks Werner M\"{u}ller and Mathematical Institute of the University of Bonn
for their hospitality during her three month research stay, where a part of this work was done.

\section{Preliminaries}\label{sect1}

\subsection{Notation}

Let $F$ be a nonarchimedean local field of characteristic zero.
Let $G_n$, $n \in \mathbb{Z}^+$, be $Sp(2n,F)$, $SO(2n+1,F)$ or  $SO(2n,F)$.
We define $G_0$ to be the trivial group.
For $G=G_n$ or $G=GL(n,F)$,
fix the minimal parabolic subgroup consisting of all upper triangular matrices in $G$
and the maximal torus consisting of all diagonal matrices in $G$.
If $\delta_1, \delta_2$ are smooth representations of $GL(m,F)$, $GL(n,F)$, respectively,
we define
$
           \delta_1 \times \delta_2   =   \Ind_{P}^G (\delta_1\otimes \delta_2)
$
where $G = GL(m+n,F)$ and $P=MU$ is the standard parabolic subgroup of $G$ with Levi factor 
$M \cong GL(m,F) \times GL(n,F)$. Similarly, if $\delta$ is a smooth representation of $GL(m,F)$
and $\sigma$ is a smooth representation of $G_n$,
we define
\[
           \delta \rtimes \sigma   =   \Ind_{P}^{G_{m+n} }(\delta \otimes \sigma)
\]
where  $P=MU$ is the standard parabolic subgroup of $G_{m+n}$ with Levi factor 
$M \cong GL(m,F) \times G_n$. 
We denote by $\mathcal{E}_2(G)$ the set of equivalence classes of irreducible square integrable 
representations of $G$ and by 
$^0\mathcal{E}(G)$ the set of equivalence classes of irreducible unitary supercuspidal 
representations of $G$.

We say that a homomorphism $h : X \to GL(d, \mathbb{C})$ is symplectic (respectively, orthogonal)
if $h$ fixes
an alternating form (respectively, a symmetric form) on $GL(d, \mathbb{C})$.
We denote by $S_a$ the standard $a$-dimensional irreducible algebraic representation of 
$SL(2, \mathbb{C})$. Then
\begin{equation}\label{Sa}
    S_a \text{ is} \begin{cases}
        \text{orthogonal}, & \text{ for } a \text{ odd,}\\
        \text{symplectic}, & \text{ for } a \text{ even.}
     \end{cases}
\end{equation}

Let $\rho$ be an irreducible supercuspidal unitary representation of $GL(d,F)$.
According to the local Langlands correspondence for $GL_d$ (\cite{HT}, \cite{Hen1}), attached to
$\rho$ is an $L$-parameter  $\varphi : W_F \to GL(d, \mathbb{C})$.
Suppose $\rho \cong \tilde{\rho}$. Then $\varphi \cong \tilde{\varphi}$
and one of the Artin $L$-functions $L(s, \rm{Sym}^2 \varphi)$ or
$L(s, \wedge^2 \varphi)$ has a pole. The $L$-function $L(s, \rm{Sym}^2 \varphi)$
has a pole if and only if $\varphi$ is orthogonal.
The $L$-function $L(s, \wedge^2 \varphi)$ has a pole if and only if $\varphi$ is symplectic.
From the result of Henniart \cite{Hen2}, we know
\begin{equation}\label{L=L}
   L(s, \wedge^2 \varphi) = L(s, \rho, \wedge^2),\, \text{and }
     L(s, \rm{Sym}^2 \varphi) = L(s, \rho, \rm{Sym}^2),
\end{equation}
where $L(s, \rho, \wedge^2)$ and $L(s, \rho, \rm{Sym}^2)$ are the Langlands $L$-functions
as defined by Shahidi \cite{Sh1}.

Let $\rho$ be an irreducible supercuspidal unitary representation of $GL(d,F)$
 and $a \in \mathbb{Z}^+$.
We define  $\delta(\rho,a)$  to be the unique irreducible 
subrepresentation of 
\[
   \rho ||^{(a-1)/2} \times \rho ||^{(a-3)/2} \times \cdots \times \rho ||^{(-(a-1))/2},
\] 
(see \cite{Zel}).
\subsection{Jordan blocks}\label{sect1.2}

We now review the definition of Jordan blocks from \cite{MoeT}.
Let $G$ be $Sp(2n,F)$, $SO(2n+1,F)$ or $O(2n,F)$. For $d \in \mathbb{N}$, 
let $r_d$ denote the standard representation of $GL(d, \mathbb{C})$.
Define
\[
    R_d = \begin{cases}
        \wedge^2 r_d, & \text{ for } G = Sp(2n,F), O(2n,F),\\
        {\rm Sym}^2 r_d, & \text{ for } G = SO(2n+1,F).
     \end{cases}
\]
Let $\sigma$ be an irreducible discrete series representation of $G_n$. 
Denote by $\Jord(\sigma)$
the set of pairs $(\rho,a)$, where $\rho \in \, ^0\mathcal{E}(GL(d_\rho,F))$, 
$\rho \cong \tilde{\rho}$,
 and $a \in \mathbb{Z}^+$,
such that 

\begin{enumerate}

\item[{(J-1)}] $a$ is even if $L(s, \rho, R_{d_\rho})$ has a pole at $s=0$ and odd otherwise,

\item[{(J-2)}] $ \delta(\rho,a) \rtimes \sigma$ is irreducible.

\end{enumerate}
For
$\rho \in \, ^0\mathcal{E}(GL(d_\rho,F))$, $\rho \cong \tilde{\rho}$,
define
\[
      \Jord_\rho(\sigma) = \{ a \mid (\rho, a) \in \Jord(\sigma) \}.
\]

Let $\hat{G}$ denote the complex dual group of $G$. Then 
$\hat{G} = SO(2n+1, \mathbb{C})$ for $G = Sp(2n,F)$,
$\hat{G} =  Sp(2n,\mathbb{C})$ for $G = SO(2n+1,F)$ and
$\hat{G} =  O(2n,\mathbb{C})$ for $G = O(2n,F)$.

\begin{lemma}\label{J-1}
 Let $\sigma$ be an irreducible discrete series representation of $G_n$. 
Let $\rho$ be an irreducible
supercuspidal self-dual representation of  $GL(d_\rho,F)$ and $a \in \mathbb{Z}^+$.
Then $(\rho,a) \in \Jord(\sigma)$ if and only if the following conditions hold

\begin{enumerate}

\item[{(J-1$'$)}] $\rho \otimes S_a$ is of the same type as $\hat{G}$,

\item[{(J-2)}] $ \delta(\rho,a) \rtimes \sigma$ is irreducible.

\end{enumerate}
\end{lemma}

\begin{proof}
We will prove (J-1) $\Leftrightarrow$ (J-1$'$).
We know from \cite{Sh2} that one and only one of the two $L$-functions
$L(s, \rho, \wedge^2)$ and $L(s, \rho, \rm{Sym}^2)$ has a pole at $s=0$.
Suppose $G = Sp(2n,F)$ or $O(2n,F)$.
We consider  $L(s, \rho, \wedge^2)$.
It has a pole at $s=0$ if and only if the parameter
$\rho: W_F \to GL(d_\rho,\mathbb{C})$ is symplectic.
According to (\ref{Sa}), this is equivalent to $\rho \otimes S_a$ being orthogonal
for $a$ even. Therfore, for $(\rho,a) \in \Jord(\sigma)$,
$a$ is even if and only if $\rho \otimes S_a$ is orthogonal.
For $G = SO(2n+1,F)$, the arguments are similar.
\end{proof}

\subsection{Assumptions}\label{sect1.3}
In this paper, we use the classification of discrete series for
classical $p$-adic groups of M{\oe}glin and Tadi\'c \cite{MoeT}, so we have to make the same 
assumptions as there.  
Let $\sigma$ be an irreducible supercuspidal representation of $G_n$ and let
$\rho$ be an irreducible self-dual supercuspidal representation of a general linear group.
We assume:
\\

(BA) \quad $\nu^{\pm(a+1)/2}\rho \rtimes \sigma$ reduces for
\[
       a = \begin{cases}
        \max \Jord_\rho(\sigma), & {\rm if}\, \Jord_\rho(\sigma) \ne \emptyset,\\
      0, &  {\rm if}\, L(s, \rho, R_{d_\rho}) {\rm \, has \, a \, pole \, at \, } s=0 
         {\rm \, and \,} \Jord_\rho(\sigma) = \emptyset,\\
    -1,  & {\rm \, otherwise},
      \end{cases}
\]
moreover, there are no other reducibility points in $\mathbb{R}$.
\\

In addition, we assume that the $L$-parameter of $\sigma$ is given by
\begin{equation}\label{Lsigma}
    \varphi_\sigma = \bigoplus_{(\rho, a) \in \Jord(\sigma)} \varphi_\rho \otimes S_a.
\end{equation}
Here, $\varphi_\rho$ denotes the $L$-parameter of $\rho$ given by \cite{HT}, \cite{Hen1}.

In \cite{Moe2}, assuming certain Fundamental Lemmas,
M{\oe}glin proved the validity of the assumptions for $SO(2n+1,F)$ and
showed how Arthur's results imply the Langlands classification of discrete series for 
$SO(2n+1,F)$.

\subsection{The Arthur $R$-group}

Let  $^LG=\hat G\rtimes W_F$ be the $L$--group of $G,$ and suppose $^LM$ is the $L$--group of a Levi subgroup, $M,$ of $G.$  Then $^LM$ is a Levi subgroup of $^LG$ (see \cite{Bor}, Section 3, for definition of parabolic subgroups and Levi subgroups of $^L\!G$).  Suppose  $\psi$ is an $A$-parameter of $G$ which factors through 
$^L\!M$,
\[
\psi: W_F \times SL(2,{\mathbb C})\times SL(2,{\mathbb C}) \longrightarrow {}^L\!M \subset \, ^L\!G.
\]

Then we can regard $\psi$ as 
an $A$-parameter of $M$.
Suppose, in addition, the image of $\psi$ is not contained in a smaller Levi subgroup
(i.e., $\psi$ is an elliptic parameter of $M$).   

Let $S_{\psi}$ be the centralizer in $\hat{G}$ of the image of $\psi$ and 
$S_{\psi}^0$ its identity component. If $T_{\psi}$ is a maximal torus of $S_{\psi}^0$, define
\[
\begin{array}{l}
W_{\psi}=N_{S_{\psi}}(T_{\psi})/Z_{S_{\psi}}(T_{\psi}), 
\\
W_{\psi}^0=N_{S_{\psi}^0}(T_{\psi})/Z_{S_{\psi}^0}(T_{\psi}),
\\
R_\psi = W_{\psi} / W_{\psi}^0.
\end{array}
\]
Lemma 2.3 of \cite{BZ} and the discussion on page 326 of \cite{BZ}
imply that $W_{\psi}$ can be identified with a subgroup of $W(G,A)$.

Let $\sigma$ be an irreducible unitary representation of $M$. 
Assume $\sigma$ belongs to the $A$-packet  $\Pi_\psi(M)$.
If $W(\sigma)=\{w \in W(G,A) \,|\, w\sigma \cong \sigma\}$, we let
\[
\begin{array}{l}
W_{\psi, \sigma}=W_{\psi} \cap W(\sigma), 
\\
W_{\psi, \sigma}^0=W_{\psi}^0 \cap W(\sigma)
\end{array}
\]
and take 
$
R_{\psi, \sigma}=W_{\psi, \sigma}/W_{\psi, \sigma}^0
$
as the Arthur R-group.

\section{Centralizers} 

Let $G$ be $Sp(2n,F)$, $SO(2n+1,F)$ or $O(2n,F)$. Let $\hat{G}$ be the complex
dual group of $G$. Let
\[
\psi: W_F \times SL(2, {\mathbb C}) \times SL(2, {\mathbb C}) \longrightarrow \hat{G}
   \subset GL(N, {\mathbb C})
\]
be an $A$-parameter. We consider $\psi$ as a representation.
Then $\psi$ is a direct sum of irreducible subrepresentations. 
Let $\psi_0$ be an irreducible subrepresentation. 
For $m \in \mathbb{N}$, set
\[
   m\psi_0 = \underbrace{\psi_0 \oplus \cdots \oplus \psi_0.}_{m-times}
\]
If 
$\psi_0 \ncong \tilde{\psi}_0$, then it can be shown
using the bilinear form on $\hat{G}$ that $\tilde{\psi}_0$ is also a subrepresentation of $\psi$.
Therefore, $\psi$  decomposes into 
a sum of irreducible subrepresentations
\[
    \psi = (m_1\psi_1 \oplus m_1\tilde{\psi}_1)\oplus \cdots \oplus 
    (m_k\psi_k \oplus m_k\tilde{\psi}_k) \oplus m_{k+1}\psi_{k+1} \oplus \cdots \oplus 
    m_\ell\psi_\ell,
\]
where $\psi_i \ncong \psi_j$, $\psi_i \ncong \tilde{\psi}_j$ for $i \ne j$. In addition,
$\psi_i \ncong \tilde{\psi}_i$ for $i=1, \dots, k$ and
$\psi_i \cong \tilde{\psi}_i$ for $i=k+1, \dots, \ell$. 
If $\psi_i \cong \tilde{\psi}_i$, then $\psi_i$ factors through a symplectic or orthogonal group. 
In this case,
if $\psi_i$ is not of the same type as $\hat{G}$, then $m_i$ must be even.
This follows again using the bilinear form on $\hat{G}$.

We want to compute $S_{\psi}$ and $W_\psi$.
First, we consider the case $\psi = m \psi_0$ or $\psi = m \psi_0 \oplus m \tilde{\psi}_0$, where $\psi_0$ is irreducible.
The following lemma is an extension of Proposition 6.5 of \cite{GrP}.
A part of the proof was communicated to us by Joe Hundley.

\begin{lemma}\label{GP}
Let $G$ be $Sp(2n,F)$, $SO(2n+1,F)$ or $O(2n,F)$.
Let
$
\psi_0: W_F \times SL(2, {\mathbb C}) \times SL(2, {\mathbb C}) \to GL(d_0, {\mathbb C})
$
be an irreducible parameter.

\begin{enumerate}

\item[{(i)}] Suppose $\psi_0 \ncong \tilde{\psi}_0$ and $\psi = m \psi_0 \oplus m \tilde{\psi}_0$.
Then $S_\psi \cong GL(m, {\mathbb C})$ and $R_\psi = 1$.

\item[{(ii)}] Suppose $\psi_0 \cong \tilde{\psi}_0$ and $\psi = m \psi_0$. Suppose $\psi_0$
is of the same type as $\hat{G}$. Then
\[
     R_\psi \cong \begin{cases}
         \mathbb{Z}_2, &  m \text{ even},\\
         1, &   m \text{ odd. }
    \end{cases}
\]

\item[{(iii)}] Suppose $\psi_0 \cong \tilde{\psi}_0$ and $\psi = m \psi_0$. Suppose $\psi_0$
is not of the same type as $\hat{G}$. Then $m$ is even,
$S_\psi \cong Sp(m, {\mathbb C})$ and $R_\psi = 1$.
 
\end{enumerate}
\end{lemma}

\begin{proof} 
(i) The proof of the statement is the same as in \cite{GrP}.

(ii) and (iii): Suppose $G=Sp(2n,F)$ or $SO(2n+1,F)$. 
Let $V$ and $V_0$ denote the spaces of the representations $\psi$ and
$\psi_0$, respectively. Denote by $\langle\,, \rangle$ the $\psi$-invariant
bilinear form on $V$ and by $\langle\, , \rangle_0$
the $\psi_0$-invariant bilinear form on $V_0$.
There exists an isomorphism $V \to V_0 \oplus \cdots \oplus V_0$.
Equivalently, $V \cong W \otimes V_0$, where $W$ is a finite dimensional vector space
with trivial $W_F \times SL(2, {\mathbb C}) \times SL(2, {\mathbb C})$-action. 
The space $W$ can be identified with 
$
{\rm Hom}_{W_F \times SL(2, {\mathbb C}) \times SL(2, {\mathbb C})}(V_0,V).
$
Then the map $W \otimes V_0 \to V$ is 
\[
     \ell \otimes v \mapsto \ell(v), \quad
   \ell \in {\rm Hom}_{W_F \times SL(2, {\mathbb C}) \times SL(2, {\mathbb C})}(V_0,V), v \in V_0.
\]
We claim there exists a nondegenerate bilinear form $\langle\, , \rangle_W$ on $W$ such that
$
   \langle\, , \rangle = \langle\, , \rangle_W \otimes \langle\, , \rangle_0
$
in the sense that
\[
   \langle \ell_1 \otimes v_1, \ell_2 \otimes v_2 \rangle = 
   \langle \ell_1, \ell_2 \rangle_W  \langle v_1, v_2\rangle_0, \quad \forall 
   \ell_1, \ell_2 \in W, \, v_1, v_2 \in V_0.
\]
The key ingredient is Schur's lemma, or rather, the variant which says that every invariant bilinear 
form on $V_0$ is a scalar multiple of $\langle\, , \rangle_0$.
Given any two 
$
 \ell_1, \ell_2 \in {\rm Hom}_{W_F \times SL(2, {\mathbb C}) \times SL(2, {\mathbb C})}(V_0,V),
$
\[
   \langle \ell_1(v_1), \ell_2(v_2)  \rangle
\]
is an invariant bilinear form on $V_0$ and therefore it is equal to $c \langle\, , \rangle_0$,
for some constant $c$. We can define $\langle \ell_1, \ell_2 \rangle_W$ by
\[
   \langle \ell_1, \ell_2 \rangle_W = 
   \frac{\langle \ell_1(v_1), \ell_2(v_2)  \rangle}{ \langle v_1, v_2\rangle_0}
\]
because Schur's lemma tells us that the right-hand side is independent of $v_1, v_2 \in V_0$.
This proves the claim. Observe that if $\psi_0$
is not of the same type as $\psi$, the form $\langle\, , \rangle_W$ is alternating,
while in the case when $\psi_0$ and $\psi$ are of the same type, the form 
$\langle\, , \rangle_W$ is symmetric.

Now, $ \im \psi =\{ I_m \otimes g \mid g \in \im \psi_0 \}$ and
\[
  \begin{aligned}
   Z_{GL(N, \mathbb{C})} (\im \psi) &= \{ g \otimes z \mid g \in GL(m, \mathbb{C}), 
         z \in \{ \lambda I_{d_0} \mid \lambda \in \mathbb{C}^\times \} \} \\
     &= \{ g \otimes I_{d_0} \mid g \in GL(m, \mathbb{C}) \}.
  \end{aligned}
\]
Let us denote by $\mathcal{W}$ the group of matrices in $GL(W)$ which preserve $\langle\, , \rangle_W$,
i.e., $\mathcal{W}= Sp(m,\mathbb{C})$ if $\langle\, , \rangle_W$ is an alternating form and
$\mathcal{W}= O(m,\mathbb{C})$ if $\langle\, , \rangle_W$ is a symmetric form.
Then
\[
   S_\psi = Z_{GL(N, \mathbb{C})} (\im \psi) \cap \hat{G} =
     \{ g \otimes I_{d_0} \mid g \in \mathcal{W}, \, \det(g \otimes I_{d_0})=1 \}.
\]
It follows that in case (iii) we have $S_\psi \cong Sp(m,\mathbb{C})$, $S_\psi^0 =S_\psi$ and 
$R_\psi = 1$.

In case (ii), $\mathcal{W}= O(m,\mathbb{C})$. Since $\det(g \otimes I_{d_0}) = (\det g)^{d_0}$,
it follows 
\[
     S_\psi \cong \begin{cases}
         O(m, {\mathbb C}), &  d_0 \text{ even, }\\
         SO(m, {\mathbb C}), &  d_0 \text{ odd }.
    \end{cases}
\]
In the case $G=SO(2n+1,F)$, $\psi_0$ is symplectic and $d_0$ is even.
Then $S_\psi \cong O(m,\mathbb{C})$ and $S_\psi^0 \cong SO(m,\mathbb{C})$.
If $m$ is even, this implies $R_\psi \cong \mathbb{Z}_2$.
For $m$ odd, $W_\psi = W_\psi^0$ and $R_\psi=1$.

In the case $G=Sp(2n,F)$, we have $ \hat{G} = SO(2n+1,\mathbb{C})$ and $md_0 = 2n+1$.
It follows that $m$ and $d_0$ are both odd. Then $S_\psi \cong SO(m,\mathbb{C})$, 
$S_\psi^0 =S_\psi$ and 
$R_\psi = 1$.

The case $G=O(2n,F)$ is similar, but simpler, because there is no condition on determinant.
 It follows $S_\psi \cong O(m,\mathbb{C})$. This implies 
$R_\psi \cong \mathbb{Z}_2$ for $m$ even and $R_\psi=1$ for $m$ odd.
\end{proof}

\begin{lemma}\label{R_psi}

Let $G$ be $Sp(2n,F)$, $SO(2n+1,F)$ or $O(2n,F)$. Let 
$
\psi: W_F \times SL(2, {\mathbb C}) \times SL(2, {\mathbb C}) \to \hat{G}
$
be an $A$-parameter. We can write $\psi$ in the form
\begin{equation}\label{psi} 
   \begin{aligned}
     \psi \cong \left( \bigoplus_{i=1}^p( m_i \psi_i \oplus m_i \tilde{\psi}_i)\right) &\oplus 
       \left( \bigoplus_{i=p+1}^q 2m_i \psi_i \right)\\  
  &\oplus 
      \left( \bigoplus_{i=q+1}^r (2m_i+1) \psi_i \right) \oplus  
     \left( \bigoplus_{i=r+1}^s 2m_i \psi_i \right)
   \end{aligned}
\end{equation}
where $\psi_i$ is irreducible, for $ i \in \{1, \dots, s \}$,
and
\[
    \begin{array}{c l}
     \psi_i \ncong \psi_j, \, \psi_i \ncong \tilde{\psi}_j, & \text{ for }
      i,j \in \{1, \dots, s \}, \, i \ne j, \\
     \psi_i \ncong \tilde{\psi}_i, & \text{ for }
      i \in \{1, \dots, p \}, \\
     \psi_i \cong \tilde{\psi}_i, & \text{ for }
      i \in \{p+1, \dots, s \},\\
     \psi_i \text{ not of the same type as }  \hat{G}, &  \text{ for }
      i \in \{p+1, \dots, q \},\\
     \psi_i \text{ of the same type as }  \hat{G}, &  \text{ for }
      i \in \{q+1, \dots, s \}.
    \end{array}
\]
Let $d=s-r$. Then
\[
      R_\psi \cong \mathbb{Z}_2^d.
\]

\end{lemma}

\begin{proof}

Set
    $\Psi_i = m_i\psi_i \oplus m_i \tilde{\psi}_i$, for $i \in \{1, \dots, p \},$
 and
    $\Psi_i = m_i\psi_i$, for $i \in \{p+1, \dots, s \}.$
Denote by $Z_i$ the centralizer of the image of $\Psi_i$ in the corresponding $GL$.
Then
\[
   Z_{GL(N, \mathbb{C})} (\im \psi)  = Z_1 \times \cdots \times Z_s
\quad \text{ and } \quad
 S_\psi = Z_{GL(N, \mathbb{C})} (\im \psi) \cap \hat{G}.
\]
Lemma~\ref{GP} tells us  the factors corresponding to $i \in \{1, \dots, q \}$ do not
contribute to $R_\psi$. In addition, we can see from the proof of Lemma~\ref{GP}
that these factors do not appear in determinant considerations.
Therefore, we can consider only the factors corresponding to $i \in \{q+1, \dots, s \}$.
 Let 
$
    \mathcal{Z}= Z_{q+1} \times \cdots \times Z_s
$
and 
$
   \mathcal{S} = \mathcal{Z} \cap \hat{G}.
$
In the same way as in the proof of Lemma~\ref{GP}, we obtain
\begin{equation}\label{E}
    \mathcal{S} \cong \left\{ (g_{q+1}, \dots, g_s) \mid \begin{array}{c}
g_i \in O(2m_i+1, \mathbb{C}), 
  i \in \{q+1, \dots, r \}, \\ g_i \in O(2m_i, \mathbb{C}),   i \in \{r+1, \dots, s \},\\
 \prod_{i=q+1}^s ( \det g_i)^{\dim \psi_i} =1
   \end{array} \right\},
\end{equation}
for $G=SO(2n+1,F)$ or $Sp(2n,F)$. For $G=O(2n,F)$, we omit the condition on determinant.
If $G=SO(2n+1,F)$, then for $i \in \{q+1, \dots, s \}$, $\psi_i$ is symplectic
and $\dim \psi_i$ is even. Therefore, the product in (\ref{E}) is always equal to 1.

Now, for $G=SO(2n+1,F)$ and $G=O(2n,F)$, we have
\[
  \mathcal{S} \cong \prod_{i=q+1}^r O(2m_i+1, \mathbb{C}) \times \prod_{i=r+1}^s O(2m_i, \mathbb{C}).
\]
It follows
$
   R_\psi \cong \prod_{i=q+1}^r 1 \times \prod_{i=r+1}^s \mathbb{Z}_2 \cong \mathbb{Z}_2^d.
$

It remains to consider $G=Sp(2n,F)$, $\hat{G}=SO(2n+1, \mathbb{C})$.
Observe that we have
\[
    \sum_{i=1}^q 2m_i \dim \psi_i   + \sum_{i=q+1}^r (2m_i+1)\dim \psi_i  
      + \sum_{i=1}^p 2m_i \dim \psi_i = 2n+1.
\]
Since the total sum is odd, we must have $r > q$ and $\dim \psi_i$ odd, for some 
$i \in \{q+1, \dots, r \}$. Without loss of generality, we may assume $\dim \psi_{q+1}$ odd.
Then
\[
  \mathcal{S} \cong SO(2m_{q+1}+1, \mathbb{C}) \times
     \prod_{i=q+2}^r O(2m_i+1, \mathbb{C}) \times \prod_{i=r+1}^s O(2m_i, \mathbb{C}).
\]
It follows
$
   R_\psi \cong 1 \times \prod_{i=q+2}^r 1 \times \prod_{i=r+1}^s \mathbb{Z}_2 \cong \mathbb{Z}_2^d.
$
\end{proof}

\section{Even orthogonal groups}\label{orthogonal}

\subsection{$R$-groups for non-connected groups}

In this section, we review some results of \cite{GH}.
Let $G$ be a reductive $F$-group. Let $G^0$ be the connected component of the identity in $G$.
We assume that $G/G^0$ is finite and abelian.

Let $\pi$ be an irreducible unitary representation of $G$.
We say that $\pi$ is discrete series if the matrix coefficients of $\pi$ are square
integrable modulo the center of $G$.

We will consider the parabolic subgroups and the $R$-groups as defined in \cite{GH}.
Let $P^0 = M^0U$ be a parabolic subgroup of $G^0$. Let $A$ be the split component in the 
center of $M^0$. Define $M = C_G(A)$ and $P=MU$. Then $P$ is called the cuspidal parabolic subgroup 
of $G$ lying over $P^0$.
The Lie algebra $\mathcal{L}(G)$ can be decomposed into root spaces with respect to 
the roots $\Phi$ of $\mathcal{L}(A)$
\[
     \mathcal{L}(G) = \mathcal{L}(M) \oplus \sum_{\alpha \in \Phi} \mathcal{L}(G)_\alpha.
\]
Let $\sigma$ be an irreducible unitary representation of $M$. 
We denote by $r_{M^0,M}(\sigma)$ the restriction of $\sigma$ to $M^0$.
Then (\cite{GH}, Lemma 2.21)
$\sigma$ is discrete series if and only if any irreducible constituent of $r_{M^0,M}(\sigma)$
is discrete series. Now, suppose $\sigma$ is discrete series. 
Let $\sigma_0$ be an irreducible constituent of $r_{M^0,M}(\sigma)$. Then $\sigma_0$
is discrete series and we have the Knapp-Stein $R$-group $R(\sigma_0)$ for $i_{G^0,M^0}(\sigma_0)$
(\cite{Kn-St, Sil2}).
We review the definition of $R(\sigma_0)$. Let
$
   W(G^0, A) = N_{G^0}(A)/M^0
$
and
$
    W_{G^0}(\sigma_0) =\{ w \in W_G(M) \mid w\sigma_0 \cong \sigma_0 \}.
$
For $w \in W_{G^0}(\sigma_0)$, we denote by $\mathcal{A}(w, \sigma_0)$ the normalized standard 
intertwining operator associated to $w$ (see \cite{Sil1}). 
Define
\[
  W_{G^0}^0(\sigma_0) =\{ w \in W_{G^0}(\sigma_0) \mid \mathcal{A}(w, \sigma_0) \text{ is a scalar}\}.
\]
Then $W_{G^0}^0(\sigma_0) = W(\Phi_1)$ is generated by reflections in a set $\Phi_1$
of reduced roots of $(G,A)$. Let $\Phi^+$ be the positive system of reduced roots of $(G,A)$
determined by $P$ and let $\Phi_1^+= \Phi_1 \cap \Phi^+$.
Then
\[
    R(\sigma_0) = \{ w \in W_{G^0}(\sigma_0) \mid w\beta \in \Phi^+ \text{ for all } 
     \beta \in \Phi_1^+ \}
\]
and $W_{G^0}(\sigma_0) = R(\sigma_0) \ltimes W(\Phi_1)$.

For definition of $R(\sigma)$, we follow \cite{GH}.
Define
$
   N_G(\sigma) =\{ g \in N_G(M) \mid g\sigma \cong \sigma \},
$
$
   W_G(\sigma) = N_G(\sigma)/M
$
and 
\[
    R(\sigma) = \{ w \in W_{G}(\sigma) \mid w\beta \in \Phi^+ \text{ for all } 
     \beta \in \Phi_1^+ \}.
\]
For $w \in W_{G}(\sigma)$, let $\mathcal{A}(w, \sigma)$ denote the intertwining operator 
on $i_{G,M}(\sigma)$ defined on page 135 of \cite{GH}.
Then the $\mathcal{A}(w, \sigma)$, $w \in R(\sigma)$, form a basis for the algebra of 
intertwining operators on $i_{G,M}(\sigma)$ (\cite{GH}, Theorem 5.16.).
In addition, $W_{G}(\sigma) = R(\sigma) \ltimes W(\Phi_1)$.
For $w \in W_{G}(\sigma)$, $\mathcal{A}(w, \sigma)$ is a scalar if and only if $w \in W(\Phi_1)$
(\cite{GH}, Lemma 5.20).

\subsection{Even orthogonal groups}

Let $G=O(2n,F)$ and $G^0=SO(2n,F)$. Then $G = G^0 \rtimes \{ 1, s\}$,
where
$
   s = \diag ( I_{n-1}, \left( \begin{matrix}
        0 & 1 \\
        1 & 0  
       \end{matrix} \right), I_{n-1})
$
and it acts on $G^0$ by conjugation.

a) Let
\[
\begin{aligned}
   M^0 & = \{ \diag (g_1, \dots, g_r, h , \, ^\tau g_r^{-1}, \dots, \, ^\tau g_1^{-1}) \mid
    g_i \in GL(n_i,F), h \in SO(2m,F) \} \\
  & \cong GL(n_1,F) \times \cdots \times GL(n_r,F) \times SO(2m,F),
\end{aligned} 
\]
where $m>1$ and $n_1 + \cdots + n_r + m =n$. Then $M^0$ is a Levi subgroup of $G^0$.
The split component of $M^0$ is
\[
   A = \{ \diag (\lambda_1 I_{n_1}, \dots, \lambda_r I_{n_r}, I_{2m} ,  
      \lambda_r^{-1} I_{n_r}, \dots, \lambda_1^{-1} I_{n_1})
\mid
    \lambda_i \in F^\times \}.
\]
Then $M= C_G(A)$  is equal to
\begin{equation}\label{M1}
\begin{aligned}
   M & = \{ \diag (g_1, \dots, g_r, h , \, ^\tau g_r^{-1}, \dots, \, ^\tau g_1^{-1}) \mid
    g_i \in GL(n_i,F), h \in O(2m,F) \} \\
  & \cong GL(n_1,F) \times \cdots \times GL(n_r,F) \times O(2m,F).
\end{aligned} 
\end{equation}
Let $\pi \in \mathcal{E}_2(M)$.
Then 
$
   \pi \cong \rho_1 \otimes \cdots \otimes \rho_k \otimes \sigma,
$
where $\rho_i \in \mathcal{E}_2(GL(n_i,F))$ and $ \sigma \in \mathcal{E}_2(O(2m,F))$.
Let 
$
   \pi_0 \cong \rho_1 \otimes \cdots \otimes \rho_k \otimes \sigma_0
$ 
be an irreducible component of $r_{M^0,M}(\pi)$.
If $s \sigma_0 \cong \sigma_0$, then $W_G(\pi)= W_{G^0}(\pi_0)$ and 
$R(\pi) = R(\pi_0)$. In this case, $r_{M^0,M}(\pi) = \pi_0$ (\cite{BJ}, Lemma 4.1)
and $\rho_i \rtimes \sigma$ is reducible
if and only if $\rho_i \rtimes \sigma_0$ is reducible (Proposition 2.2 of \cite{G2}).
Then Theorem 6.5 of \cite{G1} tells us that $R(\pi) \cong \mathbb{Z}_2^d$,
where $d$ is the number of inequivalent $\rho_i$ with $\rho_i \rtimes \sigma$ reducible.

Now, consider the case $s \sigma_0 \ncong \sigma_0$. It follows from Lemma 4.1 of 
\cite{BJ} that $\pi = i_{M,M^0}(\pi_0)$. Then $i_{G,M}(\pi) = i_{G,M^0}(\pi_0)$
and we know from Theorem 3.3 of \cite{G2} that $R(\pi) \cong \mathbb{Z}_2^d$,
 where $d=d_1+d_2$, $d_1$ is the number of inequivalent $\rho_i$
such that $n_i$ is even and $\rho_i \rtimes \sigma$ is reducible, and
$d_2$ is the number of inequivalent $\rho_i$
such that $n_i$ is odd and $\rho_i \cong \tilde{\rho}_i$.
Moreover,  Corollary 3.4 of  \cite{G2} implies if $n_i$ is odd and 
$\rho_i \cong \tilde{\rho}_i$, then 
$\rho_i \rtimes \sigma$ is reducible. Therefore, we see that 
$R(\pi) \cong \mathbb{Z}_2^d$,
where $d$ is the number of inequivalent $\rho_i$ with $\rho_i \rtimes \sigma$ reducible.

In the case $m=1$, since 
$
SO(2,F) = \{ \left( \begin{matrix}
        a & 0 \\
        0 & a^{-1}  
       \end{matrix} \right) \mid a \in F^\times \},
$
we have
\[
\begin{aligned}
   M^0 & = \{ \diag (g_1, \dots, g_r, a, a^{-1} , \, ^\tau g_r^{-1}, \dots, \, ^\tau g_1^{-1}) \mid
    g_i \in GL(n_i,F), a \in F^\times  \} \\
  & \cong GL(n_1,F) \times \cdots \times GL(n_r,F) \times GL(1,F),
\end{aligned} 
\]
and this case is described in b).

b) Let $M^0$ be a Levi subgroup of $G^0$ of the form
\[
   M^0  = \{ \diag (g_1, \dots, g_r, \, ^\tau g_r^{-1}, \dots, \, ^\tau g_1^{-1}) \mid
    g_i \in GL(n_i,F) \}   
\]
where $n_1 + \cdots + n_r =n$. 
The split component of $M^0$ is
\[
   A = \{ \diag (\lambda_1 I_{n_1}, \dots, \lambda_r I_{n_r},   
      \lambda_r^{-1} I_{n_r}, \dots, \lambda_1^{-1} I_{n_1})
\mid
    \lambda_i \in F^\times \}
\]
and $M= C_G(A) = M^0$. Therefore, 
\begin{equation}\label{M2}
\begin{aligned}
   M & = \{ \diag (g_1, \dots, g_r, \, ^\tau g_r^{-1}, \dots, \, ^\tau g_1^{-1}) \mid
    g_i \in GL(n_i,F) \} \\
  & \cong GL(n_1,F) \times \cdots \times GL(n_r,F).
\end{aligned} 
\end{equation}
Let
$
   \pi \cong \rho_1 \otimes \cdots \otimes \rho_k \otimes 1 \in \mathcal{E}_2(M),
$
where 1 denotes the trivial representation of the trivial group. Since $M=M^0$,
we can apply directly Theorem 3.3 of \cite{G2}. It follows 
$R(\pi) \cong \mathbb{Z}_2^d$,
 where $d=d_1+d_2$, $d_1$ is the number of inequivalent $\rho_i$
such that $n_i$ is even and $\rho_i \rtimes 1$ is reducible, and
$d_2$ is the number of inequivalent $\rho_i$
such that $n_i$ is odd and $\rho_i \cong \tilde{\rho}_i$.
As above, it follows from Corollary 3.4 of  \cite{G2} that if $n_i$ is odd and 
$\rho_i \cong \tilde{\rho}_i$, then 
$\rho_i \rtimes \sigma$ is reducible. 
Again, we obtain
$R(\pi) \cong \mathbb{Z}_2^d$,
where $d$ is the number of inequivalent $\rho_i$ with $\rho_i \rtimes \sigma$ reducible.

We summarize the above considerations in the following lemma.
Observe that the group $O(2,F)$ does not have square integrable representations.
It also does not appear as a factor of cuspidal Levi subgroups of $O(2n,F)$.
We call a subgroup $M$ defined by (\ref{M1}) or (\ref{M2}) a standard Levi subgroup of $O(2n,F)$.

\begin{lemma}\label{Rorth} 
Let $G = O(2n,F)$ and let 
$
  M \cong GL(n_1,F) \times \cdots \times GL(n_r,F) \times O(2m,F),
$
where $m \geq 0$, $m \ne 1$, $n_1 + \cdots + n_r + m =n$, be a standard Levi subgroup of $G$.
Let
$
   \pi \cong \rho_1 \otimes \cdots \otimes \rho_k \otimes \sigma \in \mathcal{E}_2(M).
$
Then $R(\pi) \cong \mathbb{Z}_2^d$,
where $d$ is the number of inequivalent $\rho_i$ with $\rho_i \rtimes \sigma$ reducible.

\end{lemma}

\section{$R$-groups of discrete series}

Let $G$ be $Sp(2n,F)$, $SO(2n+1,F)$ or $O(2n,F)$.

\begin{theorem}
 
Let $\pi$ be an irreducible discrete series representation of a standard Levi subgroup
$M$ of $G_n$. Let $\varphi$ be the $L$-parameter of $\pi$. Then $R_{\varphi, \pi} \cong R(\pi)$.

\end{theorem}

\begin{proof}

We can write $\pi$ in the form
\begin{equation}\label{pi}
    \pi \cong (\otimes^{m_1}\delta_1) \otimes \cdots \otimes (\otimes^{m_r}\delta_r) 
      \otimes \sigma
\end{equation}
where $\sigma$ is an irreducible discrete series representation of $G_m$ and
$\delta_i$ ($i=1, \dots, r$) is an irreducible discrete series representation of $GL(n_i,F)$
such that
$\delta_i \ncong \delta_j$,  for $i \ne j$.
As explained in Section~\ref{orthogonal}, if $G_n = O(2n,F)$, then $m \ne 1$.

Let $\varphi_i$ denote the $L$-parameter of $\delta_i$ and $\varphi_\sigma$
the $L$-parameter of $\sigma$. Then the $L$-parameter $\varphi$ of $\pi$ is
\[
     \varphi \cong ( m_1 \varphi_1 \oplus m_1 \tilde{\varphi}_1) \oplus \cdots \oplus 
        (m_r \varphi_r \oplus m_r \tilde{\varphi}_r) \oplus \varphi_\sigma .
\]
Each $\varphi_i$ is irreducible. The parameter $\varphi_\sigma$ is of the form
$
   \varphi_\sigma = \varphi_1' \oplus \cdots \oplus \varphi_s'
$
where $\varphi_i'$ are irreducible, $\varphi_i' \cong \tilde{\varphi}_i'$ and 
$\varphi_i' \ncong \varphi_i'$ for $i \ne j$.
In addition, $\varphi_i'$ factors through a group of the same type as $\hat{G}_n$.
The sets $\{ \varphi_i \mid i=1, \dots, r \}$ and $\{ \varphi_i' \mid i=1, \dots, s \}$
can have nonempty intersection.
After rearranging the indices, we can write $\varphi$ as
\[
 \begin{aligned}
     \varphi \cong \left( \bigoplus_{i=1}^h( m_i \varphi_i \oplus m_i \tilde{\varphi}_i)\right) \oplus 
       \left( \bigoplus_{i=h+1}^q 2m_i \varphi_i \right) &\oplus 
     \left( \bigoplus_{i=q+1}^k 2m_i \varphi_i \right)\\
  &\oplus 
      \left( \bigoplus_{i=k+1}^r (2m_i+1) \varphi_i \right) \oplus  
      \left( \bigoplus_{i=r+1}^\ell \varphi_i \right),
 \end{aligned}
\]
where
$
    \varphi_\sigma = \bigoplus_{i=k+1}^\ell \varphi_i
$
and
\[ 
    \begin{array}{c l}
     \varphi_i \ncong \varphi_j, \, \varphi_i \ncong \tilde{\varphi}_j, & \text{ for }
      i,j \in \{1, \dots, \ell\}, \, i \ne j, \\
     \varphi_i \ncong \tilde{\varphi}_i, & \text{ for }
      i \in \{1, \dots, h \}, \\
     \varphi_i \cong \tilde{\varphi}_i, & \text{ for }
      i \in \{h+1, \dots, \ell \},\\
     \varphi_i \text{ not of the same type as }  \hat{G}, &  \text{ for }
      i \in \{h+1, \dots, q \},\\
     \varphi_i \text{ of the same type as }  \hat{G}, &  \text{ for }
      i \in \{q+1, \dots, k \}.
    \end{array}
\]
Let $d = k-q$. Lemma~\ref{R_psi} implies $R_\varphi \cong \mathbb{Z}_2^d$.
In addition, $R_{\varphi, \pi} \cong R_\varphi$.

On the other hand, we know  that $R(\pi) \cong \mathbb{Z}_2^c$,
where $c$ is cardinality of the set
\[
    C = \{ i \in \{1, \dots, r \} \mid \delta_i \rtimes \sigma \text{ is reducible} \}.
\]
This follows from \cite{G1}, for $G=SO(2n+1,F),$  $Sp(2n,F),$ and from Lemma~\ref{Rorth} 
for $G = O(2n,F)$.
We want to show $C = \{q+1, \dots, k \}$. For any $i \in \{1, \dots, \ell \}$,
$\varphi_i$ is an irreducible representation of $W_F \times SL(2, \mathbb{C})$
and therefore it can be written in the form
$
     \varphi_i = \varphi_i' \otimes S_{a_i},
$
where $\varphi_i'$ is an irreducible representation of $W_F$ and $S_{a_i}$ is the 
standard irreducible $a_i$-dimensional algebraic representation of $SL(2, \mathbb{C})$.
For $i \in \{1, \dots, r \}$, this parameter corresponds to the representation $\delta( \rho_i, a_i)$.
Therefore, 
the representation $\delta_i$ in (\ref{pi}) is 
$
  \delta_i = \delta( \rho_i, a_i).
$
From (\ref{Lsigma}), we have
\[
    \varphi_\sigma = \bigoplus_{i=k+1}^\ell \varphi_i = 
          \bigoplus_{(\rho, a) \in \Jord(\sigma)} \varphi_\rho \otimes S_a.
\]
For $i \in \{h+1, \dots, q \}$, $\varphi_i$ is not of the same type as $\hat{G}$
and $ \delta(\rho_i,a_i) \rtimes \sigma$ is irreducible.
For $i \in \{q+1, \dots, k \}$, $\varphi_i$ is of the same type as $\hat{G}$.
Now, Lemma~\ref{J-1} tells us $(\rho_i,a_i) \in \Jord(\sigma)$ if and only if
$ \delta(\rho_i,a_i) \rtimes \sigma$ is irreducible.
Therefore, $ \delta(\rho_i,a_i) \rtimes \sigma$ is irreducible for $i \in \{k+1, \dots, r \}$
and $ \delta(\rho_i,a_i) \rtimes \sigma$ is reducible for $i \in \{q+1, \dots, k \}$.
It follows $C = \{q+1, \dots, k \}$ and $R(\pi) \cong \mathbb{Z}_2^d \cong R_{\varphi, \pi}$,
finishing the proof.
\end{proof}

\section{Unitary groups}\label{sect5}

Let $E/F$ be a quadratic extension of $p$-adic fields. Fix $\theta \in W_F \setminus W_E$.
Let 
$G=U(n) $ be a unitary group defined with respect to $E/F$, $U(n) \subset GL(n,E)$.
Let
\[
J_n =
\small
        \left( \begin{matrix}
                     & & & & 1 \\
                     & & & -1 &  \\
                     & & 1 & &  \\
                     & \cdot & & &  \\
                     \cdot & & & &  \\
                \end{matrix} \right)
\normalsize
\quad {\rm and} \quad
\tilde{J}_n =
\small
        \left( \begin{matrix}
                     & & & & 1 \\
                     & & & 1 &  \\
                     & & 1 & &  \\
                     & \cdot & & &  \\
                     \cdot & & & &  \\
                \end{matrix} \right)
\normalsize
\]
We have
\[
         ^L\! G = GL(n, \mathbb{C}) \rtimes W_F,
\]
where $W_E$ acts trivially on $GL(n, \mathbb{C})$ and 
the action of $w \in W_F \setminus W_E$  on $g \in GL(n, \mathbb{C})$ is
given by 
$
       w(g) = J_n \,^tg^{-1} J_n^{-1}.
$

\subsection{$L$-parameters for Levi subgroups}\label{sect5.1}

Suppose we have a Levi subgroup $M \cong Res_{E/F}GL_k \times U(\ell)$.
Then 
\[
       ^L\!M^0 = \{ \left(  \begin{matrix} g & &  \\
                                         & m & \\  & & h    \end{matrix} \right) \mid
                g, h \in GL(k, \mathbb{C}), m \in  GL(\ell, \mathbb{C})  \}.
\]
Direct computation shows that the action of $w \in W_F \setminus W_E$ on $^L\!M^0$  is given by
\[
             w( \left(  \begin{matrix} g & &  \\
                                         & m & \\  & & h    \end{matrix} \right) ) =
\left(  \begin{matrix} J_k \,^th^{-1} J_k^{-1} & &  \\
                     & J_\ell \,^tm^{-1} J_\ell^{-1} & \\  & & J_k \,^tg^{-1} J_k^{-1}    \end{matrix} \right).
\]

Let $\pi$ be a discrete series representation of  $GL(k,E) = (Res_{E/F}GL_k)(F)$
and $\tau$ a discrete series representation of $U(\ell)$. 
Let $\varphi_\pi : W_E \times SL(2, \mathbb{C}) \to GL(k, \mathbb{C})$ 
be the $L$-parameter of $\pi$ and 
$\varphi_\tau : W_F \times SL(2, \mathbb{C}) \to GL(\ell, \mathbb{C}) \rtimes W_F$ the $L$-parameter of $\tau$.
Write
\[
        \varphi_\tau (w,x) = (\varphi'_\tau (w,x), w), \quad w \in W_F, x \in SL(2, \mathbb{C}).
\]

According to \cite{Bor}, sections 4, 5 and 8, there exists a unique (up to equivalence) $L$-parameter
$
  \varphi: W_F \times SL(2, \mathbb{C}) \to \, ^L\!M
$
such that
\begin{equation}\label{w}
      \begin{aligned}
         \varphi((w,x)) &= (\varphi_\pi(w), *, *,w),  \quad \forall w \in W_E,  x \in SL(2, \mathbb{C}),\\
          \varphi((w,x)) &= (*, \varphi'_\tau (w,x),  *,w),  \quad \forall w \in W_F, x \in SL(2, \mathbb{C}).
     \end{aligned}
\end{equation}
We will define a map 
$
  \varphi: W_F\times SL(2, \mathbb{C}) \to \, ^L\!M
$
satisfying (\ref{w}) and show that $\varphi$ is a homomorphism. 
Define
\begin{equation}\label{DefWE}
   \varphi((w,x)) = (\varphi_\pi(w,x), \varphi'_\tau (w,x), ^t\!\varphi_\pi(\theta w\theta^{-1},x)^{-1},w), 
\quad w \in W_E,  x \in SL(2, \mathbb{C})
\end{equation}
and
\[
     \varphi((\theta,1)) = (J_k^{-1}, \varphi'_\tau (\theta,1), ^t\!\varphi_\pi(\theta^2,1)^{-1}J_k, \theta).
\]
Note that
\[
      \begin{aligned}
           \varphi_\tau (\theta^2, 1) &= ( \varphi'_\tau (\theta,1), \theta) ( \varphi'_\tau (\theta,1), \theta)\\
               &= ( \varphi'_\tau (\theta,1), 1) (J_\ell \,^t\varphi'_\tau (\theta,1)^{-1} J_\ell^{-1}, \theta^2)\\
               &= (\varphi'_\tau (\theta,1)J_\ell \,^t\varphi'_\tau (\theta,1)^{-1} J_\ell^{-1}, \theta^2).
     \end{aligned}
\]
It follows
\begin{equation}\label{tau}
     \varphi'_\tau (\theta,1)J_\ell \,^t\varphi'_\tau (\theta,1)^{-1} J_\ell^{-1} = \varphi_\tau' (\theta^2, 1).
\end{equation}
Similarly, for $w \in W_E$, $x \in SL(2, \mathbb{C})$,
\[
      \begin{aligned}
           \varphi_\tau(\theta w \theta^{-1}, x) 
&=
       \varphi_\tau(\theta, 1) \varphi_\tau(w, x) \varphi_\tau(\theta,1)^{-1}   \\
&=
       (\varphi_\tau'(\theta, 1), \theta)  (\varphi_\tau'(w, x),w) (1, \theta^{-1})
        (\varphi_\tau'(\theta,1)^{-1} ,1)  \\
&=
       (\varphi_\tau'(\theta, 1), 1) (J_\ell \,^t\varphi'_\tau (w,x)^{-1} J_\ell^{-1},\theta w \theta^{-1})
        (\varphi_\tau'(\theta,1)^{-1} ,1)  \\
&=
     (\varphi_\tau'(\theta, 1)J_\ell \,^t\varphi'_\tau (w,x)^{-1} J_\ell^{-1}\varphi_\tau'(\theta,1)^{-1},
        \theta w \theta^{-1})
    \end{aligned}
\]
and thus
\begin{equation}\label{tau2}
     \varphi_\tau'(\theta, 1)J_\ell \,^t\varphi'_\tau (w,x)^{-1} J_\ell^{-1}\varphi_\tau'(\theta,1)^{-1}
=
    \varphi_\tau'(\theta w \theta^{-1},x).
\end{equation}
Now,
\[
      \begin{aligned}
           \varphi (\theta,1) \varphi (\theta,1)
&=
       (J_k^{-1}, \varphi'_\tau (\theta,1), ^t\!\varphi_\pi(\theta^2, 1)^{-1}J_k, \theta)
      (J_k^{-1}, \varphi'_\tau (\theta,1), ^t\!\varphi_\pi(\theta^2, 1)^{-1}J_k, \theta)    \\
&=
       (J_k^{-1}, \varphi'_\tau (\theta,1), ^t\!\varphi_\pi(\theta^2, 1)^{-1}J_k, 1)
   (J_k\varphi_\pi(\theta^2,1), J_\ell \, ^t\!\varphi'_\tau (\theta,1)^{-1}J_\ell^{-1}, J_k^{-1}, \theta^2) \\
&= 
      (\varphi_\pi(\theta^2,1), \varphi_\tau' (\theta^2, 1), ^t\!\varphi_\pi(\theta^2,1)^{-1}, \theta^2) \\
&= 
       \varphi (\theta^2, 1),
    \end{aligned}
\]
using (\ref{tau}) and (\ref{DefWE}). Further, for $w \in W_E$, $x \in SL(2, \mathbb{C})$, we have
\[
      \begin{aligned}
          \varphi(\theta,1) &\varphi(w,x) \varphi(\theta,1)^{-1}   \\ 
&= 
       (J_k^{-1}, \varphi'_\tau (\theta,1), ^t\!\varphi_\pi(\theta^2,1)^{-1}J_k, \theta)
      (\varphi_\pi(w,x), \varphi'_\tau (w,x), ^t\!\varphi_\pi(\theta w\theta^{-1},x)^{-1},w)  \\
&\hspace{2cm}
      \cdot (1,1,1,\theta^{-1})
      (J_k, \varphi'_\tau (\theta,1)^{-1}, J_k^{-1}\, ^t\!\varphi_\pi(\theta^2,1), 1)
\\ 
&= 
       (J_k^{-1}, \varphi'_\tau (\theta,1), ^t\!\varphi_\pi(\theta^2,1)^{-1}J_k, 1)
 \\
&\hspace{2cm}
  \cdot (J_k \varphi_\pi(\theta w\theta^{-1},x)J_k^{-1}, J_\ell \,^t\varphi'_\tau (\theta,x)^{-1} J_\ell^{-1}, 
     J_k \, ^t\!\varphi_\pi(w,x)^{-1} J_k^{-1}, \theta w \theta^{-1})
  \\
&\hspace{4cm}
      \cdot (J_k, \varphi'_\tau (\theta,1)^{-1}, J_k^{-1}\,^t\!\varphi_\pi(\theta^2,1), 1)
\\ 
&= 
      ( \varphi_\pi(\theta w\theta^{-1},x), \varphi_\tau' (\theta w\theta^{-1},x),
^t\!\varphi_\pi(\theta^2 w\theta^{-2},x)^{-1},\theta w\theta^{-1}) 
\\ 
&= 
       \varphi(\theta w\theta^{-1},x).
         \end{aligned}
\]
Here, we use (\ref{tau2}) and $J_k^2 = (J_k^{-1})^2 = (-1)^{k-1}$, so
\[
       ^t\!\varphi_\pi(\theta^2,1)^{-1}J_k J_k \, ^t\!\varphi_\pi(w,x)^{-1} J_k^{-1} 
   J_k^{-1}\,^t\!\varphi_\pi(\theta^2,1) = \, ^t\!\varphi_\pi(\theta^2 w\theta^{-2},x)^{-1}.
\]
In conclusion, we have $\varphi(\theta^2, 1) = \varphi(\theta, 1)^2$ and
$
    \varphi(\theta w\theta^{-1},x) = \varphi(\theta,1) \varphi(w,x) \varphi(\theta,1)^{-1}.
$
Since $\varphi$ is clearly multiplicative on $W_E\times SL(2, \mathbb{C})$, 
it follows that $\varphi$ is a homomorphism.
Therefore, $\varphi$ is the $L$-parameter for $\pi \otimes \tau$.

\subsection{Coefficients $\lambda_\varphi$}
Let $\varphi:W_E \times SL(2,\mathbb{C}) \to GL_k(\mathbb{C})$ be an irreducible $L$-parameter.
Assume
$\varphi \cong \,^t(^\theta \varphi)^{-1}$.
Let $X$ be a nonzero matrix such that
\[
      ^t\varphi(\theta w\theta^{-1}, x)^{-1} = X^{-1} \varphi(w,x) X,
\]
for all $w \in W_E$, $x \in SL(2,\mathbb{C})$.
We proceed similarly as in \cite{Moe}, p.190. By taking transpose and inverse, 
\[
      \varphi(\theta w\theta^{-1}, x) = \, ^tX \, ^t \varphi (w, x)^{-1} \, ^tX^{-1}.
\]
Next, we replace $w$ by $\theta w\theta^{-1}$. This gives
\[
     \varphi(\theta^2,1)  \varphi(w, x) \varphi(\theta^{-2},1) 
                = \, ^tX \, ^t \varphi (\theta w\theta^{-1},x)^{-1} \, ^tX^{-1}
= \, ^tX X^{-1} \varphi(w,x) X \, ^tX^{-1},     
\]
for all $w \in W_E$, $x \in SL(2,\mathbb{C})$.
Since $\varphi$ is irreducible, $ \varphi(\theta^{-2},1)\,^tX X^{-1}$ is a constant. 
Define
\begin{equation}\label{LdbPhi}
       \lambda_\varphi = \varphi(\theta^{-2},1)\,^tX X^{-1}.
\end{equation}
As in \cite{Moe}, we can show that $\lambda_\varphi = \pm 1$.

\begin{lemma}

Let $\varphi:W_E  \to GL_k(\mathbb{C})$ be an irreducible $L$-parameter
such that $\varphi \cong \,^t(^\theta \varphi)^{-1}$.
Let $S_a$ be the standard $a$-dimensional irreducible algebraic representation of 
$SL(2, \mathbb{C})$. 
Then 
$
     ^\theta (\, ^t(\varphi\otimes S_a)^{-1} ) \cong \varphi\otimes S_a
$
and
\[
        \lambda_{\varphi \otimes S_a}  = (-1)^{a+1} \lambda_\varphi.
\]

\end{lemma}

\begin{proof} We know that $^tS_a^{-1} \cong  S_a$. Let $Y$ be a nonzero matrix such that
\[
        ^tS_a(x)^{-1} = Y^{-1} S_a(x) Y,
\]
for all $x \in SL(2,\mathbb{C})$.
Then $^tY = Y$ for $a$ odd and $^tY = -Y$ for $a$ even.
Let $X$ be a nonzero matrix such that
\[
      ^t\varphi(\theta w\theta^{-1})^{-1} = X^{-1} \varphi(w) X,
\]
for all $w \in W_E$. We have
\[
\begin{aligned}
       ^t(\varphi\otimes S_a(\theta w\theta^{-1}, x))^{-1}
&=
     (^t\varphi(\theta w\theta^{-1})^{-1} )\otimes (^tS_a(x)^{-1} ) \\
&= 
       ( X^{-1} \varphi(w) X )  \otimes ( Y^{-1} S_a(x) Y)  \\
&= (X \otimes Y)^{-1} (\varphi \otimes S_a(w,x))  \otimes (X \otimes Y).
\end{aligned}
\]
It follows that
$^\theta (\, ^t(\varphi\otimes S_a)^{-1} ) \cong \varphi\otimes S_a$
and
\[
   \begin{aligned}
        \lambda_{\varphi \otimes S_a}  
&= 
                    ( \varphi \otimes S_a(\theta^{-2},1)) \,^t(X \otimes Y) (X \otimes Y)^{-1}\\
&=
           (\varphi(\theta^{-2})\,^tX X^{-1}) \otimes (\,^tY Y^{-1}))
=
(-1)^{a+1} \lambda_\varphi.
\end{aligned}
\]
\end{proof}

\subsection{Centralizers}

Let $\varphi : W_F \times SL(2,\mathbb{C}) \to \,^L G$ be an $L$-parameter. 
Denote by $\varphi_E$ the restriction 
of $\varphi$ to $W_E\times SL(2,\mathbb{C})$. Then $\varphi_E$ is a representation of 
$W_E\times SL(2,\mathbb{C})$ on $V=\mathbb{C}^n$.
Write $\varphi_E$ as a sum of irreducible subrepresentations
\[
     \varphi_E = m_1 \varphi_1 \oplus \cdots \oplus m_\ell \varphi_\ell,
\]
where $m_i$ is the multiplicity of $\varphi_i$ and 
$\varphi_i \ncong \varphi_j$, for $i \ne j$. It follows from \cite{Moe} that $S_\varphi$,
the centralizer in $\hat{G}$ of the image of $\varphi$, is given by
\begin{equation}\label{Sphi}
   S_\varphi  \cong  \prod_{i=1}^\ell C(m_i \varphi_i),
\end{equation}
where
\[
       C(m_i \varphi_i) = \begin{cases}
                 GL(m_i,\mathbb{C}), &{\rm if} \, \varphi_i \ncong \,^\theta \widetilde{\varphi_i }, 
\\
                 O(m_i,\mathbb{C}), &{\rm if} \, \varphi_i  \cong \,^\theta \widetilde{\varphi_i }, \, 
                                     \lambda_{\varphi_i} = (-1)^{n-1},
\\
                 Sp(m_i,\mathbb{C}), &{\rm if} \, \varphi_i  \cong \,^\theta \widetilde{\varphi_i }, \, 
                                     \lambda_{\varphi_i} = (-1)^{n}.
     \end{cases}
\]

\subsection{Coefficients $\lambda_\rho$}

Let $^LM = GL_k(\mathbb{C}) \times GL_k(\mathbb{C}) \rtimes W_F$,
where the action of $w \in W_F \setminus W_E$ on $GL_k(\mathbb{C}) \times GL_k(\mathbb{C})$
is given by
$
   w (g, h, 1) w^{-1} = (J_n \,^th^{-1} J_n^{-1}, J_n \,^tg^{-1} J_n^{-1}, 1).
$
For $\eta = \pm1$, we denote by $R_\eta$ the representation of $^LM$ on 
${\rm End}_\mathbb{C}(\mathbb{C}^k)$ given by 
\[
 \begin{aligned}
    R_\eta((g,h,1)) \cdot X &= g Xh^{-1},\\
    R_\eta((1,1,\theta)) \cdot X &= \eta \tilde{J}_k \,^t X \tilde{J}_k.
 \end{aligned}
\]
Let $\tau$ denote the nontrivial element in $Gal(E/F)$.
Let $\rho$ be a supercuspidal discrete series representation of $GL(k, E)$.
Assume $\rho \cong \,^\tau\widetilde{\rho}$.
Then precisely one of the two $L$-functions $L(s, \rho, R_1)$ and $L(s, \rho, R_{-1})$
has a pole at $s=0$. Denote by $\lambda_\rho$ the value of $\eta$ such that $L(s, \rho, R_\eta)$
has a pole at $s=0$.

\begin{lemma}

Let $\rho$ be a supercuspidal discrete series representation of $GL(k, E)$ such that 
$\rho \cong \,^\tau\widetilde{\rho}$. 
Assume $k$ is odd.
Let $\varphi_\rho$ be the $L$-parameter of $\rho$.
Then 
$
  \lambda_{\varphi_\rho} = \lambda_\rho.
$

\end{lemma}

\begin{proof}

As shown in Section~\ref{sect5.1}, the parameter $\varphi : W_F \to \, ^LM$ corresponding to 
$\varphi_\rho : W_E \to GL_k(\mathbb{C})$ is given by
\begin{equation}\label{WE}
   \varphi(w) = \left( \left(    \begin{matrix}
\varphi_\rho(w) & \\
               &  ^t\varphi_\rho(\theta w\theta^{-1})^{-1}
        \end{matrix}      \right), w \right),
\end{equation}
for $w \in W_E$, and
\begin{equation}\label{ThE}
     \varphi(\theta) = \left( \left(    \begin{matrix}
          J_k^{-1}  &  \\
                   & ^t \varphi_\rho(\theta^2)^{-1}J_k
        \end{matrix}      \right),  \theta \right).
\end{equation}
From \cite{Hen2}, we have $L(s, \rho, R_\eta) = L(s, R_\eta \circ \varphi)$.
Therefore, $L(s, R_{\lambda_\rho} \circ \varphi)$ has a pole at $s=0$. 
Then $R_{\lambda_\rho} \circ \varphi$ contains
the trivial representation, so there exists nonzero $X \in M_k(\mathbb{C})$ such that
$(R_{\lambda_\rho} \circ \varphi)(w) \cdot X = X$, for all $w \in W_F$.
In particular, (\ref{WE}) implies that for $w \in W_E$, 
\[
        \varphi_\rho(w) X \, ^t\varphi_\rho(\theta w\theta^{-1}) = X
\]
so
\begin{equation}\label{X}
         \varphi_\rho(w) X = X \, ^t\varphi_\rho(\theta w\theta^{-1})^{-1}.
\end{equation}
Therefore, $X$ is a nonzero intertwining operator between  $\varphi_\rho$ and 
$^t(^\theta \varphi_\rho)^{-1}$. From (\ref{LdbPhi}), we have
\begin{equation}\label{Xlambda}
    \varphi_\rho(\theta^{-2}) \,^tX X^{-1} = \lambda_{\varphi_\rho}.
\end{equation}
Now, since $(R_{\lambda_\rho} \circ \varphi)(\theta) \cdot X = X$, we have from (\ref{ThE})
\[
        J_k^{-1} \tilde{J}_k \,^t X \tilde{J}_k
                  J_k^{-1} \, ^t \varphi_\rho(\theta^2) = \lambda_\rho X.
\]
By taking determinant, we get $\det (\varphi_\rho(\theta^2)) =\lambda_\rho^k$. The  equation (\ref{Xlambda})
implies (again by taking determinant)
\[
       \lambda_{\varphi_\rho}^k = \lambda_\rho^k.
\]
Since  $k$ is odd, we have $\lambda_{\varphi_\rho} = \lambda_\rho$. 
\end{proof}

\subsection{Jordan blocks for unitary groups}

For the unitary group $U(n)$, define
\[
         R_d = R_\eta, \quad {\rm where} \quad \eta = (-1)^n.
\]
Let $\sigma$ be an irreducible discrete series representation of $U(n)$. 
Denote by $Jord(\sigma)$
the set of pairs $(\rho,a)$, where $\rho \in \, ^0\mathcal{E}(GL(d_\rho,E))$, 
$\rho \cong \,^\tau\tilde{\rho}$,
 and $a \in \mathbb{Z}^+$,
such that $(\rho,a)$ satisfies properties (J-1) and (J-2) from  Section~\ref{sect1.2}.

\begin{lemma} Let $\rho$ be an irreducible supercuspidal representation of $GL(d,E)$
such that
$
\varphi_\rho \cong \,^\theta \widetilde{\varphi}_\rho,
$
where $\varphi_\rho$ is the $L$-parameter for $\rho$. If $d$ is even, assume 
$\lambda_{\varphi_\rho} = \lambda_\rho$.
Then the condition (J-1) is equivalent to

\begin{enumerate}

\item[{(J-1$''$)}] $ \lambda_{{\varphi_\rho}  \otimes S_a}  = (-1)^{n+1} .$

\end{enumerate}
\end{lemma}

\begin{proof}
The condition (J-1) says that 
$a$ is even if $L(s, \rho, R_d)$ has a pole at $s=0$ and odd otherwise.
Observe that 
\[
     \begin{aligned}
           L(s, \rho, R_d)  \, {\rm has \, a \, pole \, at \,} s=0
&\Leftrightarrow
           \lambda_{\varphi_\rho} = (-1)^n \\
&\Leftrightarrow
           \lambda_{{\varphi_\rho}  \otimes S_a}  = (-1)^n(-1)^{a+1}\\
&\Leftrightarrow
           \lambda_{{\varphi_\rho}  \otimes S_a}  = \begin{cases}
                    (-1)^{n+1}, &a \, {\rm even},\\ 
                     (-1)^n , &a \, {\rm odd} .
          \end{cases}\hspace{2.5cm}
     \end{aligned}
\]
From this, it is clear that (J-1) is equivalent to (J-1$''$).
\end{proof}

\subsection{$R$-groups for unitary groups}

\begin{lemma}

Let $\sigma$ be an irreducible discrete series representation of $U(n)$ and
let $\delta= \delta(\rho, a)$ be an irreducible discrete series representation of $GL(\ell,E)$,
$\ell= da$, $d = \dim(\rho)$. 
Let $\varphi_\rho$ and $\varphi$ be the $L$-parameters of $\rho$  and
$\pi = \delta \otimes \sigma$, respectively. If $d$ is even and  
$
\varphi_\rho \cong \,^\theta \widetilde{\varphi}_\rho,
$
assume $\lambda_{\varphi_\rho} = \lambda_\rho$.
Then 
$
     R_{\varphi, \pi} \cong R(\pi).
$

\begin{proof}

Let $\varphi_\sigma$ be the $L$-parameter of $\sigma$. 
Then
\[
    \varphi_E \cong \varphi_\rho \otimes S_a \oplus \,^\theta \widetilde{\varphi}_\rho \otimes S_a
    \oplus (\varphi_\sigma)_E.
\]
This is a representation of $W_E \times SL(2,\mathbb{C})$ on $V =\mathbb{C}^{n+2\ell}$.
Write $(\varphi_\sigma)_E$ as a sum of irreducible components,
\[
     (\varphi_\sigma)_E =  \varphi_1 \oplus \cdots \oplus  \varphi_m.
\]
Each component appears with multiplicity one.
The centralizer $S_\varphi$ is given by (\ref{Sphi}). If 
$
\varphi_\rho \ncong \,^\theta \widetilde{\varphi}_\rho,
$
then
\[
       S_\varphi \cong GL(1, \mathbb{C}) \times GL(1, \mathbb{C}) \times \prod_{i=1}^m GL(1, \mathbb{C}).
\]
This implies $R_{\varphi} = 1$. On the other hand, $\delta \rtimes \sigma$ is irreducible, so 
$R(\pi) = 1$. It follows 
$
     R_{\varphi, \pi} \cong R(\pi).
$

Now, consider the case 
$
\varphi_\rho \cong \,^\theta \widetilde{\varphi}_\rho.
$
If 
$
  \varphi_\rho \otimes S_a  \in \{  \varphi_1, \dots , \varphi_m\},
$
then 
\[
       S_\varphi \cong O(3, \mathbb{C})  \times \prod_{i=1}^{m-1} GL(1, \mathbb{C})
\quad  {\rm and} \quad
           S_\varphi^0 \cong SO(3, \mathbb{C})  \times \prod_{i=1}^{m-1} GL(1, \mathbb{C}).
\]
This gives $W_\varphi = W_\varphi^0$ and $R_{\varphi} = 1$. Since
$
  \varphi_\rho \otimes S_a  \in \{  \varphi_1, \dots , \varphi_m\},
$
the condition (J-2) implies that $\delta \rtimes \sigma$ is irreducible. Therefore,
$R(\pi) = 1 = R_{\varphi, \pi}$.

It remains to consider the case 
$
\varphi_\rho \cong \,^\theta \widetilde{\varphi}_\rho
$
and
$
  \varphi_\rho \otimes S_a  \notin \{  \varphi_1, \dots , \varphi_m\}.
$
Then $(\rho,a)$ does not satisfy (J-1$''$) or (J-2).
Assume first that $(\rho,a)$ does not satisfy (J-1$''$). Then  
$\delta \rtimes \sigma$ is irreducible, so $R(\pi) = 1$.
Since $(\rho,a)$ does not satisfy (J-1$''$), we have 
$
   \lambda_{\varphi_\rho \otimes S_a} = (-1)^n = (-1)^{n + 2 \ell}.
$
Then, by (\ref{Sphi}), 
\[
       S_\varphi \cong Sp(2, \mathbb{C})  \times \prod_{i=1}^m GL(1, \mathbb{C}).
\]
It follows $R_{\varphi, \pi} = 1 = R(\pi)$.

Now, assume that $(\rho,a)$ satisfies (J-1$''$), but does not satisfy (J-2). Then
$
   \lambda_{\varphi_\rho \otimes S_a} = (-1)^{n-1} = (-1)^{n +2\ell -1}, 
$
so
\[
       S_\varphi \cong O(2, \mathbb{C})  \times \prod_{i=1}^m GL(1, \mathbb{C})
\]
and $R_{\varphi, \pi} \cong \mathbb{Z}_2$.
Since  $(\rho,a)$ does not satisfy (J-2), $\delta \rtimes \sigma$ is reducible and hence
$R(\pi) \cong \mathbb{Z}_2 \cong R_{\varphi, \pi}$.
\end{proof}

\end{lemma}


\begin{thebibliography}{wwww}

\bibitem{Art1}
{J. Arthur},
{Unipotent automorphic representations: conjectures}, {\it Ast\'erisque}, {\bf 171-2}(1989), {13-71}.

\bibitem{Art2}
{J. Arthur}, {On elliptic tempered characters}, {\it Acta Math.} {\bf 171} (1993), {73-138}


\bibitem{BJ}
{D. Ban and C. Jantzen},
{Degenerate principal series for even-orthogonal groups}
{\it Represent. Theory}, {\bf 7}(2003), 440-480.


\bibitem{BZ}
{D. Ban and Y. Zhang},
{Arthur $R$-groups, classical $R$-groups, and Aubert involutions for $SO(2n+1)$}, 
{\it Compositio Math.},  {\bf 141}(2005), 323-343.

\bibitem{Bor} A.Borel, Automorphic $L$-functions, in Automorphic Forms,
Representations, and $L$-functions, Part 2, 
{\em Proc. Symp. Pure Math.} {\bf 33} (1979), 27-61.

\bibitem{G1}
{D. Goldberg}, 
{Reducibility of induced representations for $Sp(2n)$ and $SO(n)$}
{\it Amer. J. Math.}, {\bf 116}(1994), {1101-1151}.

\bibitem{G2}
{D. Goldberg}, 
{\it Reducibility for non-connected $p$-adic groups, with $G^0$ of prime index},
Can. J. Math. 47, No.2, (1995)  344-363.

\bibitem{G3} 
{D. Goldberg},
{\it On dual $R$--groups for classical groups}
{in}
{On Certain $L$--functions: Proceedings of a Conference in Honor of Freydoon Shahidi's 60th Birthday},
{AMS-Clay Math. Inst.}, {to appear}


\bibitem{GH}
{D. Goldberg and R. Herb},
{\it Some results on the admissible representations of non-connected reductive $p$-adic groups},
{Ann. Sci. Norm. Sup}. (4) {\bf 30},  (1997) {97-146.}

\bibitem{GrP}
{B. Gross and D. Prasad},
{On the decomposition of a representation of  $SO_n$ when restricted to $SO_{n-1}$}, 
{\it Can. J. Math.},  {\bf 44}(1992), 974-1002.

\bibitem{HT} M.Harris and R.Taylor, On the geometry and cohomology of
some simple Shimura varieties. {\em Ann. of Math. Studies},Vol{\bf 151},
Princeton University Press, 2001.




\bibitem{Hen1} G.Henniart, Une preuve simple des conjectures de Langlands 
pour $GL(n)$ sur un corps $p$-adique, {\em Inv. Math.} {\bf 139} (2000),
no.2, 439-455.



\bibitem{Hen2} G.Henniart, Correspondance de Langlands et 
fonctions $L$ des carres exterieur et symetrique, 
{\em Prepublications Mathematiques de l'IHES} (2003).

\bibitem{keys}
{C.D. Keys}, {$L$-indistinguishability and $R$-groups for quasi split groups: unitary groups in even
dimension}, {\it Ann. Sci. Ec. Norm. Sup.}, {\bf 20} (1987), {31-64}

\bibitem{Kn-St}
{A. Knapp and E. Stein},
{Intertwining operators for semisimple groups},
{\it Ann. of Math.}, {\bf 93}(1971), {489-578}.

\bibitem{Moe}
C. M\oe glin,
Sur la classification des s\'eries discr\`etes des groupes classiques $p$-adiques:
param\`etres de Langlands et exhaustivit\'e,
{\it J. Eur. Math. Soc.} {\bf 4} (2002) 143-200.

\bibitem{Moe2} 
C. M\oe glin,
Classification des s\'eries discr\`etes pour certains groupes classiques $p$-adiques. Harmonic analysis, group representations, automorphic forms and invariant theory, 209–245,
Lect. Notes Ser. Inst. Math. Sci. Natl. Univ. Singap., 12, World Sci. Publ., Hackensack, NJ, 2007.

\bibitem{Moe3} 
C. M\oe glin, 
Classification et changement de base pour les s\'eries discr\`etes des groupes unitaires $p$-adiques,
{\it Pacific J. Math.} 233 (2007), no. 1, 159-204. 


\bibitem{MoeT}
 C. M\oe glin and M. Tadi\'c,  
{\it Construction of discrete series for classical $p$-adic groups},
 { J. Am. Math. Soc.} 15, No.3, (2002) 715-786.



\bibitem{Sh1} {F. Shahidi}, {On certain $L$-functions}, 
 {\it Amer. J. Math.}, \textbf{103}(1981), {297-355}.

\bibitem{Sh2}
{F. Shahidi},
{A proof of Langlands conjecture on Plancherel measure; complementary
series for $p$-adic groups}
{\it Ann. of Math.}, {\bf 132}(1990), {273-330}.

\bibitem{Shel} {D. Shelstad}, {$L$--indistinguishability for real groups}, {\it Math., Ann.}, {\bf 259} (1982),  {385-430}

\bibitem{Sil1} A.Silberger, Introduction to harmonic analysis on
reductive $p$-adic groups, {\em Math. Notes} {\bf 23}, Princeton University 
Press, Princeton, NJ (1979).

\bibitem{Sil2} A.Silberger, The Knapp-Stein dimension theorem
for $p$-adic groups, {\em Proc. Amer. Math. Soc.} {\bf 68} (1978)
243-246.

\bibitem{Zel} 
{A.V. Zelevinsky},
{Induced representations of reductive pÐadic groups II, on irreducible representations of $GL(n)$},
{\it Ann. Sci. \'Ecole Norm. Sup (4)}, {\bf 13} (1980), {165-210}.




\end{thebibliography}
\end{document}